\documentclass{amsart}
\usepackage{mathtools}
\usepackage{amsfonts,amssymb,amsthm,amsmath}
\usepackage{mathrsfs}
\usepackage{enumerate}
\usepackage{pgf}
\usepackage{url}
\usepackage[all]{xy}
\SelectTips{cm}{}
\usepackage{graphicx}
\usepackage{hyperref}
\usepackage{lineno,hyperref}
\usepackage{cleveref}
\usepackage{float}
\usepackage{comment}
\usepackage{tikz}
\usetikzlibrary{shapes.geometric}
\allowdisplaybreaks

\newcommand{\R}{\mathbb{R}}
\renewcommand{\a}{\mathbf{a}}

\renewcommand{\P}{\mathcal{P}}

\newtheorem{theorem}{Theorem}
\crefname{theorem}{Theorem}{Theorems}

\crefname{proposition}{Proposition}{Propositions}
\newtheorem{lemma}[theorem]{Lemma}
\crefname{lemma}{Lemma}{Lemmas}

\crefname{corollary}{Corollary}{Corollaries}

\crefname{conjecture}{Conjecture}{Conjectures}
\theoremstyle{definition}

\crefname{remark}{Remark}{Remarks}
\newtheorem{assumption}{Assumption}
\crefname{assumption}{Assumption}{Assumptions}
\newtheorem{definition}{Definition}
\crefname{definition}{Definition}{Definitions}

\crefname{example}{Example}{Examples}
\crefname{table}{Table}{Tables}
\crefname{section}{Section}{Sections}
\crefname{equation}{Equation}{Equations}
\crefname{figure}{Figure}{Figures}
\crefname{appendix}{Appendix}{Appendices}

\title[Explicit Calculations for Sono's Sieve]{Explicit Calculations for Sono's Multidimensional Sieve of $E_2$-Numbers}

\author[Daniel Goldston, Apoorva Panidapu, and Jordan Schettler]{Daniel A. Goldston, Apoorva Panidapu, and Jordan Schettler}
\date{}

\begin{document}
\maketitle

\begin{abstract}
    We derive explicit formulas for integrals of certain symmetric polynomials used in Keiju Sono's multidimensional sieve of $E_2$-numbers, i.e., integers which are products of two distinct primes.
    We use these computations to produce the currently best-known bounds for gaps between multiple $E_2$-numbers.
    For example, we show there are infinitely many occurrences of four $E_2$-numbers within a gap size of $94$ unconditionally and within a gap size of $32$ assuming the Elliott-Halberstam conjecture for primes and sifted $E_2$-numbers.
\end{abstract}

\section{Introduction}

Let $p_n$ denote the $n$th prime.
The prime number theorem implies the average of the gap $p_{n+1} - p_n$ between successive primes is asymptotic to $\log p_n$.
On the other hand, the still unproven twin prime conjecture asserts that the gap $p_{n+1}-p_n=2$ occurs infinitely often.
The groundbreaking development in \cite{Gold0} by Goldston, Pintz, and Y{\i}ld{\i}r{\i}m of the GPY sieve, a variant of the Selberg sieve, established that there are always gaps much smaller than average:
\[\liminf_{n\rightarrow \infty} \frac{p_{n+1}-p_n}{\log p_n} = 0.\]
In fact, assuming the Elliott-Halberstam conjecture for primes (denoted $\mathrm{BV}[1,\P]$ as below), they also showed that
\[\liminf_{n\rightarrow \infty} (p_{n+1}-p_n)\leq 16.\]
Yitang Zhang was the first to establish an unconditional result on finite gaps between pairs of primes.
By using ideas from the GPY sieve along with proving a sufficiently weakened version of $\mathrm{BV}[1,\P]$, he showed in \cite{Zhan} that
\[\liminf_{n\rightarrow \infty} (p_{n+1}-p_n)\leq 70\,000\,000.\]
In general, for any positive integer $\nu$ we define
\[H_\nu \mathrel{\mathop :}= \liminf_{n\rightarrow \infty}(p_{n+\nu}-p_n).\]
Shortly after Zhang's result, Maynard (in \cite{Mayn}) and Tao (unpublished work) developed a multidimensional GPY sieve and showed that $H_\nu$ was finite for all $\nu$ with $H_1\leq 600$ unconditionally.
The Polymath8 Project further refined these techniques to obtain the currently best-known unconditional estimate $H_1\leq 246$ in \cite{Poly}.

% Given an integer $n>1$ with prime factorization $n=p_{k_1}^{a_1}p_{k_2}^{a_2}\cdots p_{k_m}^{a_m}$, the \emph{exponent pattern} (or \emph{prime signature}) of $n$ is the multiset $\{a_1, a_2, \ldots, a_m\}$ where repetition is allowed but order is irrelevant.
% For example, the exponent pattern of a prime $p=p^1$ is just $\{1\}$, and likewise the exponent pattern of a prime power $p^a$ is just $\{a\}$.
% Products of $j$ distinct primes are known as $E_j$-numbers and have exponent pattern 
% \[\{\underbrace{1,1, 1, \ldots, 1}_{\mbox{$j$ copies of $1$}}\}.\]
% More generally, if we are given a nonempty multiset of positive integers $\mathcal{E}$, we say $n$ is an $\mathcal{E}$-number if $n$ has exponent pattern $\mathcal{E}$.
% Let $q_n$ denote the $n$th $\mathcal{E}$-number ordered by inequality.
% Define the gap between $\nu+1$ successive $\mathcal{E}$-numbers by
% \[H_{\nu}^{\mathcal{E}} = \liminf_{n\rightarrow \infty}(q_{n+\nu}-q_n).\]
% The still unproven twin prime conjecture is then equivalent to
% \[H_{1}^{\{1\}} = 2,\]
% i.e., there are infinitely many pairs of primes with a gap of two.
% Yitang Zhang was the first to show that $H_{1}^{\{1\}}$ was finite by using a modification of the GPY sieve in \cite{Zhan}. Zhang's bound of 70 million was subsequently improved by Maynard \cite{Mayn} to $H_{1}^{\{1\}}\leq 600$ and further to $H_{1}^{\{1\}}\leq 246$ by the Polymath8 Project \cite{Poly}.

The GPY sieve methods can establish even stronger results for gaps between $E_2$-numbers (products of two distinct primes), which are more plentiful than prime numbers.
Let $q_n$ denote the $n$th $E_2$-number, i.e.,
\[q_1=6, q_2=10, q_3=14, q_4=15, \ldots,\]
and let $G_{\nu}$ denote the gap for $\nu+1$ successive $E_2$-numbers:
\[G_{\nu} \mathrel{\mathop :}= \liminf_{n\rightarrow \infty}(q_{n+\nu}-q_n).\]
We say that an $E_2$-number is ``sifted'' when it takes the form $p_{k_1}p_{k_2}$ where the prime factors satisfy the inequalities $N^{\eta} < p_{k_1} \leq N^{1/2} < p_{k_2}$ for arbitrary constants $\eta$ and $N$,  with $0 < \eta \leq 1/4$ and $N$ a large positive integer.
Now let $\tilde{q}_n$ denote the $n$th sifted $E_2$-number and let $\widetilde{G}_{\nu}$ denote the gap for $\nu+1$ such sifted $E_2$-numbers as $N$ becomes large:
\[\widetilde{G}_{\nu} = \lim_{N\rightarrow \infty}\liminf_{n\rightarrow \infty}(\tilde{q}_{n+\nu}-\tilde{q}_n).\]
Sifted $E_2$-numbers have conjectured gaps comparable to prime numbers, while unsifted $E_2$-numbers should have smaller gaps.
The Hardy-Littlewood $k$-tuple conjectures imply $\widetilde{G}_\nu = H_\nu = H(\nu+1)$ and $G_\nu = G(\nu+1)$ where $H(k)$ (respectively $G(k)$) is the smallest diameter of an admissible (respectively $E_2$-admissible) $k$-tuple.
The definitions for admissible and $E_2$-admissible are given in the section below.
Goldston, Graham, Pintz, and Y{\i}ld{\i}r{\i}m showed in \cite{Gold1} that we have finite gaps $G_\nu \leq \widetilde{G}_{\nu} \leq \nu e^{\nu-\gamma}(1+o(1))$ for all $\nu$ where $\gamma$ is Euler's constant, and they also established the explicit bound $\widetilde{G}_1 \leq 6$.

Following Maynard's ideas, Sono developed a multidimensional sieve for sifted $E_2$-numbers in \cite{Sono} and showed that $\widetilde{G}_2\leq 12$ under the assumption of the Elliott-Halberstam conjecture for primes $\mathrm{BV}[1,\mathcal{P}]$ and sifted $E_2$-numbers $\mathrm{BV}[1,\widetilde{\mathcal{E}}_2]$.
The present authors in \cite{Gold3} were able to obtain this result of Sono's using only the one-dimensional GGPY sieve and further proved that $\widetilde{G}_2\leq 32$ unconditionally.
Sono's multidimensional sieve, however, should be able to produce results that could not be obtained from the one-dimensional sieve.
The purpose of the present paper is to prove the following theorem using Sono's sieve, which gives \cref{E2table} showing the currently best-known gaps for $E_2$-numbers. In particular, the last two columns from the third row and beyond are entirely new, while the first two theoretical columns are provided for comparison.

\begin{theorem}\label{maintheorem}
We have the following table for gaps between $\nu+1$ successive $E_2$-numbers where the Elliott-Halberstam column indicates that we are assuming both $\mathrm{BV}[1,\mathcal{P}]$ and $\mathrm{BV}[1,\widetilde{\mathcal{E}}_2]$ there.
\begin{table}[ht]
 \setlength{\tabcolsep}{3pt}
    \centering
    \begin{tabular}{c|cccc}
        $\nu$ & Hardy-Littlewood & Hardy-Littlewood & Elliott-Halberstam & Unconditional  \\ \hline 
        \vspace{-0.35 cm} & & & &\\
        1 & $G_1=G(2)=1$ & $\widetilde{G}_1=H(2)=2$  & $\widetilde{G}_1\leq H(3)=6$    & $\widetilde{G}_1\leq H(3)=6$ \\
        2 & $G_2=G(3)=2$ & $\widetilde{G}_2=H(3)=6$  & $\widetilde{G}_2\leq H(5)=12$   & $\widetilde{G}_2\leq H(10)=32$ \\
        3 & $G_3=G(4)=4$ & $\widetilde{G}_3=H(4)=8$  & $\widetilde{G}_3\leq H(10)=32$  & $\widetilde{G}_3\leq H(23)=94$ \\
        4 & $G_4=G(5)=5$ & $\widetilde{G}_4=H(5)=12$ & $\widetilde{G}_4\leq H(16)=60$  & $\widetilde{G}_4\leq H(49)=240$ \\
        5 & $G_5=G(6)=6$ & $\widetilde{G}_5=H(6)=16$ & $\widetilde{G}_5\leq H(25)=110$ & $\widetilde{G}_5\leq H(102)=576$ \\
        6 & $G_6=G(7)=8$ & $\widetilde{G}_6=H(7)=20$ & $\widetilde{G}_6\leq H(37)=168$ & $\widetilde{G}_6\leq H(225) = 1440$ \\
    \end{tabular}
    \vspace{0.35 cm}
    \caption{Bounds for Gap Sizes of $\nu+1$ Successive $E_2$-numbers}
    \label{E2table}
\end{table}
\end{theorem}
The proof of \cref{maintheorem} will follow from Sono's \cref{sonotheorem} combined with an intricate analysis of certain integrals $I_k(F)$, $J_k(F)$, $L_k(F)$, $M_k(F)$ done via extensions of the explicit computations for $I_k(F)$, $J_k(F)$ from Section 8 in \cite{Mayn} to $L_k(F)$, $M_k(F)$. See \cref{Maynard,Glem,tildetheorem} below.

\section{Theorems of Maynard and Sono}

In order to state the results of the multidimensional sieves of Maynard and Sono, we first need to state some definitions and results about the distribution of primes and $E_2$-numbers.

Let $\pi(x;q)$ denote the number of primes which are relatively prime to $q$ and less than or equal to $x$, and let $\pi(x;q,a)$ denote the number of primes which are congruent to a unit $a$ modulo $q$ and less than or equal to $x$.
Then $\pi(x;q,a)\sim \pi(x;q)/\varphi(q)$ as $x \rightarrow \infty$, and a level of distribution for primes describes the average error in this asymptotic over large intervals.
More precisely,
if we take $\pi^\flat(x;q) = \pi(2x;q)-\pi(x;q)$ and $\pi^\flat(x;q,a)=\pi(2x;q,a)-\pi(x;q,a)$, then we say the primes have a level of distribution $\theta \in (0,1]$ if the following statement holds.
\begin{assumption}[{$\mathrm{BV}[\theta, \mathcal{P}]$}]
For all $\varepsilon>0$ and $A>0$, we have
\[\mathop{\sum_{q\leq x^{\theta-\varepsilon}}}_{q \,\, \mathrm{squarefree}} \max_{(a,q)=1} \left|\pi^\flat(x;q,a)-\frac{\pi^\flat(x;q)}{\varphi(q)} \right| \ll_A \frac{x}{\log^A x}   \]
as $x\rightarrow \infty$.
\end{assumption}
The Bombieri-Vinogradov theorem \cite{Bomb,Vino} states that $\mathrm{BV}[1/2, \mathcal{P}]$ holds, while the unproven Elliott-Halberstam conjecture for primes \cite{Elli} states that $\mathrm{BV}[1, \mathcal{P}]$ holds. 
To state the analogs for $E_2$-numbers, we define $\widetilde{\pi}_2(x;q)$ to be the number of sifted $E_2$-numbers which are relatively prime to $q$ and less than or to equal to $x$, and take $\widetilde{\pi}_2(x;q,a)$ to be the number of sifted $E_2$-numbers which are congruent to a unit $a$ modulo $q$ and less than to equal to $x$.
As above, we take $\widetilde{\pi}_2^\flat(x;q)=\widetilde{\pi}_2(2x;q)-\widetilde{\pi}_2(x;q)$ and $\widetilde{\pi}_2^\flat(x;q,a)=\widetilde{\pi}_2(2x;q,a)-\widetilde{\pi}_2(x;q,a)$.
Then we say the sifted $E_2$-numbers have a level of distribution $\theta\in (0,1]$ when the following statement holds.
\begin{assumption}[{$\mathrm{BV}[\theta, \widetilde{\mathcal{E}}_2 ]$}]
For all $\varepsilon>0$ and $A>0$, we have
\[\mathop{\sum_{q\leq x^{\theta-\varepsilon}}}_{q \,\, \mathrm{squarefree}} \max_{(a,q)=1} \left|\widetilde{\pi}_2^\flat(x;q,a)-\frac{\widetilde{\pi}_2^\flat(x;q)}{\varphi(q)} \right| \ll_A \frac{x}{\log^A x}   \]
as $x\rightarrow \infty$.
\end{assumption}

Motohoshi in \cite{Moto} proved that $\mathrm{BV}[1/2,\widetilde{\mathcal{E}}_2]$ holds, and we expect that the Elliott-Halberstam conjecture for sifted $E_2$-numbers $\mathrm{BV}[1,\widetilde{\mathcal{E}}_2]$ holds here as well.
Note that our definitions for level of distributions $\theta$ follow Sono but vary from other authors.
In particular, we can always take $\theta=1/2$ unconditionally and $\theta=1$ is equivalent to assuming the corresponding Elliott-Halberstam conjecture.

Let $\mathcal{H}=\{h_1, h_2, \ldots, h_k\}$ denote a set of $k$ distinct nonnegative integers, and let $n$ denote a positive integer variable.
In order for all $n+h_i$ to be prime simultaneously infinitely often there cannot be a prime $p$ such that the $h_i$ cover all congruence classes modulo $p$ since then we would always have at least one of the $n+h_i$ as a multiple of $p$.
We say $\mathcal{H}$ is \emph{admissible} if for all primes $p$ there is a positive integer $n_0$ (depending on $p$) such that $p$ does not divide any $n_0+h_i$.
Similarly, we say $\mathcal{H}$ is \emph{$E_2$-admissible} if for all pairs of primes $p_{k_1}\leq p_{k_2}$ there is a positive integer $n_0$ (depending on $p_{k_1}p_{k_2}$) such that $p_{k_1}p_{k_2}$ does not divide any $n_0+h_i$.
The Hardy-Littlewood $k$-tuple conjectures imply that if $\mathcal{H}$ is admissible (respectively $E_2$-admissible), then there are infinitely many $n$ such that $n+h_1$,  $n+h_2$, $\ldots$,  $n+h_k$ are all primes/sifted $E_2$-numbers (respectively unsifted $E_2$-numbers).
We call $h_k-h_1$ the \emph{diameter} of $\mathcal{H}$, and the smallest diameter for an admissible (respectively $E_2$-admissible) set of size $k$ is denoted $H(k)$ (respectively $G(k)$).
For example, $\mathcal{H}=\{0,2,6,8\}$ is admissible with minimal diameter $H(4)=8$, so we expect there are infinitely many positive integers $n$ such that
$n$, $n+2$, $n+6$, $n+8$ are all prime.
Thus we should have $H_3 = H(3+1)=8$ according to Hardy-Littlewood, but the best-known unconditional result is that $H_3\leq 24\,797\,814$ \cite{Poly}.
In general, sieve methods allow us to conclude that for sufficiently large $k$ at least $\nu+1$ of the $n+h_i$ are all primes/sifted $E_2$-numbers infinitely often for $\mathcal{H}$ admissible.
The parity problem inherent in most sieve methods, however, restricts how small $k$ can be for a given $\nu$.
In particular, we must have $k\geq 2\nu+1$ in these cases, and frequently the optimal $k=2\nu+1$ is not attainable.
For example, when $\nu=1$, we can get $\nu+1=2$ of the $n+h_i$ as sifted $E_2$-numbers infinitely often for admissible sets $\mathcal{H}$ of size $k=2\nu+1=3$, but for primes we can take $k=3$ only by assuming a generalized Elliott-Halberstam conjecture.
Likewise, when $\nu=2$, we can take $k=2\nu+1=5$ for sifted $E_2$-numbers only by assuming the Elliott-Halberstam conjecture for primes and $E_2$-numbers, and we can take $k=10$ with no such assumptions \cite{Sono,Gold3}. Hence we get $\widetilde{G}_2 \leq H(5)=12$ conditionally, and we can currently conclude $\widetilde{G}_2 \leq H(10)=32$ unconditionally.

Now we will outline the basic ideas in the sieve methods of Maynard-Tao and Sono. Let $\chi_{\mathbb{S}}(n)$ denote the characteristic function of some set $\mathbb{S}$ of interest like the set of primes $\mathcal{P}$ or sifted $E_2$-numbers $\widetilde{\mathcal{E}}_2$.
Fix a positive integer $\nu$ and an admissible set $\mathcal{H} = \{h_1, h_2, \ldots, h_k\}$ as above.
Consider the sum
\[S(x)=\sum_{x< n \leq 2x} \left(\sum_{i=1}^k \chi_{\mathbb{S}}(n+h_i) - \nu \right)w_n\]
where $w_n$ are nonnegative weights (depending on $\mathcal{H}$ and $n$).
Note that if $S(x)>0$, then we must have at least one $n$ with $x<n\leq 2x$ such that $\nu+1$ or more of the $n+h_i$ are in $\mathbb{S}$.
The standard Selberg weights are of the form
\[w_n = \left(\sum_{d|\prod_{i=1}^k (n+h_i)} \lambda_d \right)^2,\]
while the Maynard-Tao sieve weights give more flexibility by allowing weights to depend on divisors of individual factors $n+h_i$:
\[w_n = \left(\sum_{d_i|n+h_i \forall i} \lambda_{d_1,d_2,\ldots,d_k} \right)^2.\]
Here the $\lambda_{d_1,d_2,\ldots,d_k}$ are given in terms of a smooth function $F\colon \R^k\rightarrow \R$ supported on the region $\mathcal{R}_k$ of points $(x_1, x_2, \ldots, x_k)$ where $x_i \geq 0$ for all $i$ and $x_1 + x_2 + \cdots + x_k \leq 1$.
The sum $S(x)$ is decomposed as $S=S_2-\nu S_1$ where
\[S_1(x) = \sum_{x< n \leq 2x} w_n \]
and
\[S_2(x) = \sum_{x< n \leq 2x} \left(\sum_{i=1}^k \chi_{\mathbb{S}}(n+h_i)\right)w_n.\]
Then one estimates $S_1$, $S_2$ by certain iterated integrals of $F$ in order to conclude that $S_2(x)/S_1(x)>\nu$ for all sufficiently large $x$.
%Both Maynard and Sono obtain small gaps between primes and $E_2$-numbers, respectively, by computing iterated integrals of smooth functions.
Here we consider a symmetric polynomial $P$ in $k$ variables and take $F$ to have the shape
\[F(x_1,x_2,\ldots, x_k) = \left\{ \begin{array}{ll}
P(x_1, x_2, \ldots, x_k) & \mbox{if } (x_1, x_2, \ldots, x_k) \in \mathcal{R}_k\\
0 & \mbox{otherwise.}
\end{array} \right.\]
Now define quantities $I_k(F)$, $J_k(F)$:
\begin{align*}
    I_k(F) &\mathrel{\mathop:}= \int_0^1 \cdots \int_0^1 F^2 \,dx_1\cdots dx_k, \\
    J_k(F) &\mathrel{\mathop:}= \int_0^1 \cdots \int_0^1 \left( \int_0^1 F \, dx_1 \right)^2 \,dx_2\cdots dx_k.
\end{align*}
\begin{theorem}[Maynard, \cite{Mayn}]
Let $\mathcal{H}=\{h_1,h_2,\ldots,h_k\}$ be an admissible set, and suppose $\theta$ is a level of distribution for the primes.
If $\nu > 0$ is an integer such that
\[\frac{\frac{\theta}{2}kJ_k(F)}{I_k(F)}>\nu\]
for some $F$ as above, then there are infinitely many integers $n$ such that at least $\nu + 1$ of the $n+h_i$ are prime.
\end{theorem}
%Sono used the multi-dimensional methods of Maynard to prove an analogous result for sifted $E_2$-numbers.
To state the analogous result proved by Sono for sifted $E_2$-numbers, we need to define two more integral quantities $L_k(F)$, $M_k(F)$ which depend on $\theta$ and the parameter $\eta$ mentioned above:
\begin{align*}
    L_k(F) &\mathrel{\mathop:}= \int_{\eta}^{\theta/2} \frac{\theta/2-\xi}{\xi (1-\xi)} \int_0^1 \cdots \int_0^1 \left( \int_0^1 F_{\xi} \, dx_1 \right) \left( \int_0^1 F \, dx_1 \right) \,dx_2\cdots dx_k d\xi \\
    M_k(F) &\mathrel{\mathop:}= \int_{\eta}^{\theta/2} \frac{(\theta/2-\xi)^2}{\xi (1-\xi)} \int_0^1 \cdots \int_0^1 \left( \int_0^1 F_{\xi} \, dx_1 \right)^2 \,dx_2\cdots dx_k d\xi 
\end{align*}
where
\[F_{\xi}(x_1, x_2, \ldots, x_k) = F\left(\frac{2\xi}{\theta} + \left(1-\frac{2\xi}{\theta}\right)x_1, x_2, x_3, \ldots, x_k\right).\]
\begin{theorem}[Sono, \cite{Sono}]\label{sonotheorem}
Let $\mathcal{H}=\{h_1,h_2,\ldots,h_k\}$ be an admissible set, and suppose $\theta$ is a common level of distribution for the primes and sifted $E_2$-numbers. If $\nu > 0$ is an integer such that
\begin{equation}\label{sono}
\lim_{\eta\rightarrow 0^+} \frac{-\theta k L_k(F) + \frac{\theta^2}{4} \log \left(\frac{1 - \eta}{\eta}\right) k J_k(F) +  k M_k(F)}{\frac{\theta}{2} I_k(F)} > \nu
\end{equation}
for some $F$ as above, then there are infinitely many integers $n$ such that at least $\nu + 1$ of the $n+h_i$ are sifted $E_2$-numbers.
\end{theorem}

\section{Explicit Computations}

Let $P_i = x_1^i + x_2^i + \cdots + x_k^i$ denote the $i$th power sum symmetric polynomial in $k$ variables.
Maynard computed $I_k(F)$ and $J_k(F)$ explicitly when $P$ has a special form.
\begin{assumption}\label{assume}
Suppose the symmetric polynomial $P$ is of the form
\[P = \sum_{i=0}^n a_i (1-P_1)^{b_i}P_2^{c_i}\]
where the $a_i$ are real constants and $b_i$, $c_i$ are nonnegative integers.
\end{assumption}
Central to these calculations are the beta function identity $\int_0^1 (1-u)^bu^c\, du = b!c!/(b+c+1)!$ and certain polynomials $Q_c(x)$ of degree $c$ with $Q_0(x) = 1$ and for $c>1$ are given by
\[Q_c(x) = c! \sum_{r=1}^c \binom{x}{r} \mathop{\sum_{c_1,\ldots,c_r}}_{c_1+\cdots+c_r=c} \prod_{i=1}^r \frac{(2c_i)!}{c_i!}. \]
For example, $Q_1(x) = 2x$, $Q_2(x) = 20x+4x^2$, $Q_3(x) = 592x+120x^2+8x^3$, $Q_4(x)=33888x+5936x^2+480x^3+16x^4$.
\begin{theorem}[Maynard, \cite{Mayn}]\label{Maynard}
Under \cref{assume}, we have
\[I_k(F) = \sum_{0\leq i,j \leq n} a_i a_j \frac{(b_i+b_j)!Q_{c_i+c_j}(k)}{(k+b_i+b_j+2c_i+2c_j)!} \]
and
\[J_k(F) = \sum_{0\leq i,j \leq n} a_i a_j \sum_{c_1'=0}^{c_i}\sum_{c_2'=0}^{c_j}\binom{c_i}{c_1'}\binom{c_j}{c_2'}\frac{\gamma_{i,j,c_1',c_2'} Q_{c_1'+c_2'}(k-1)}{(k+b_i+b_j+2c_i+2c_j+1)!}\]
where
\[\gamma_{i,j,c_1',c_2'} = \frac{b_i!b_j!(2c_i-2c_1')!(2c_j-2c_2')!(b_i+b_j+2c_i+2c_j-2c_1'-2c_2'+2)!}{(b_i+2c_i-2c_1'+1)!(b_j+2c_j-2c_2'+1)!}.\]
\end{theorem}
Sono did not derive explicit calculations for $L_k(F)$ and $M_k(F)$ as Maynard did for $I_k(F)$ and $J_k(F)$, but we will do so here.
First, we need a lemma.
\begin{lemma}\label{Glem}
We have
     \begin{equation}\label{first}
        \mathop{\idotsint}_{\mathcal{R}_k} (1-P_1)^b P_2^c \, dx_1\cdots d_k = \frac{b! Q_c(k)}{(k+2c+b)!} 
     \end{equation}
     and
     \begin{equation}\label{second}
         \mathop{\idotsint}_{\mathcal{R}_k} P_1^b P_2^c \, dx_1\cdots d_k = \frac{Q_c(k)}{(k+2c-1)!(k+2c+b)}.
     \end{equation}
\end{lemma}
\begin{proof}
\cref{first} was shown in \cite{Mayn}, and we will use this to prove \cref{second}. We start with the binomial theorem and then simplify by applying the identity $\sum_{j=m}^n(-1)^j\binom{n}{j} = (-1)^m\binom{n-1}{n-m}$ for $0<m\leq n$ %(finite induction and Pascal's rule)
to get
\begin{align*}
    & \mathop{\idotsint}_{\mathcal{R}_k} (1-(1-P_1))^b P_2^c \, dx_1\cdots d_k \\
    =& \sum_{b_1=0}^b(-1)^{b_1}\binom{b}{b_1}\mathop{\idotsint}_{\mathcal{R}_k} (1-P_1)^{b_1} P_2^c \, dx_1\cdots d_k \\
    &= \sum_{b_1=0}^b(-1)^{b_1} \binom{b}{b_1}\frac{b_1! Q_c(k)}{(k+2c+b_1)!} \\
     &= Q_{c}(k)\frac{b!(-1)^{k}}{(k+2c+b)!}\sum_{b_1=0}^b(-1)^{k+2c+b_1}\binom{k+2c+b}{k+2c+b_1} \\
     &= Q_{c}(k)\frac{b!}{(k+2c+b)!}\binom{k+2c+b-1}{b} \\
     &=Q_c(k)\frac{1}{(k+2c-1)!(k+2c+b)}.
\end{align*}
\end{proof}
\begin{definition}
Now define limits
\[\widetilde{L}_k(F) = \lim_{\eta\rightarrow 0^+} \left(L_k(F) - \frac{\theta}{2} \log\left( \frac{1-\eta}{\eta(\frac{2}{\theta} - 1)}\right) J_k(F)\right)\]
and 
\[\widetilde{M}_k(F) = \lim_{\eta\rightarrow 0^+} \left(M_k(F) - \frac{\theta^2}{4} \log\left( \frac{1-\eta}{\eta(\frac{2}{\theta} - 1)}\right) J_k(F)\right).\]
Then the limit of the numerator on the left-hand side of the inequality in \cref{sono} becomes
\begin{align*}
    \widetilde{J}_k(F) &\mathrel{\mathop :}= \lim_{\eta\rightarrow 0^+} \left(-\theta k L_k(F) + \frac{\theta^2}{4} \log \left(\frac{1 - \eta}{\eta}\right) k J_k(F) +  k M_k(F)\right) \\
    &= -\theta k \widetilde{L}_k(F) + \frac{\theta^2}{4} \log \left(\frac{2}{\theta}-1\right) k J_k(F) +  k \widetilde{M}_k(F).
\end{align*}
\end{definition}
\begin{theorem}\label{tildetheorem}
Under \cref{assume}, we have
\begin{gather*}
\widetilde{L}_k(F) = \sum_{0\leq i,j \leq n} a_i a_j \sum_{c_1'=0}^{c_i}\sum_{c_2'=0}^{c_j}\binom{c_i}{c_1'}\binom{c_j}{c_2' } Q_{c_1' + c_2'}(k - 1) \\
 \cdot \left(\sum_{d_1=0}^{\tilde{d}}  \frac{(-1)^{d_1}\binom{\tilde{d}}{d_1}\delta_{i, j, c_1', c_2'}\lambda_{k - 1 + d_1 + 2 c_1' + 2 c_2'} }{(k + 2 c_1' + 2 c_2' - 2)! (k - 1 + 2 c_1' + 2 c_2' + d_1)} \right.\\
\left.-\sum_{b_1'=0}^{b_i}\sum_{e_1=0}^{\tilde{e}} \frac{(-1)^{e_1 + b_1'}\binom{\tilde{e}}{e_1}\binom{b_i}{b_1'}\varepsilon_{i, j, c_1', c_2', b_1'}\mu_{b_1' + 2 c_i - 2c_1' + 1, k - 1 + e_1 + 2 c_1' + 2 c_2'}}{(k + 2 c_1' + 2 c_2' - 2)!(k - 1 + 2 c_1' + 2 c_2' + e_1)} \right)
 \end{gather*}
and
\begin{gather*}
\widetilde{M}_k(F) =\frac{\theta}{2} \sum_{0\leq i,j \leq n} a_i a_j \sum_{c_1'=0}^{c_i}\sum_{c_2'=0}^{c_j}\binom{c_i}{c_1'}\binom{c_j}{c_2' }Q_{c_1' + c_2'}(k - 1)  \\
 \cdot\left(\sum_{d_1=0}^{\tilde{d}} \frac{ (-1)^{d_1} \binom{\tilde{d}}{d_1}\delta_{i, j, c_1', c_2'} \lambda_{k - 1 + d_1 + 2 c_1' + 2 c_2'}}{(k + 2 c_1' + 2 c_2' - 2)! (k - 1 + 2 c_1' + 2 c_2' + 
         d_1)} \right. \\ 
  \left.-2\sum_{b_1'=0}^{b_i} \sum_{e_1=0}^{\tilde{e}} \frac{(-1)^{e_1 + b_1'}\binom{\tilde{e}}{e_1}\binom{b_i}{b_1'} \varepsilon_{i, j, c_1', c_2', b_1'} \mu_{b_1' + 2 c_i - 2 c_1' + 1, k - 1 + e_1 + 2 c_1' + 2 c_2'}}{(k + 2 c_1' + 2 c_2' - 2)! (k - 1 + 2 c_1' + 2 c_2' + e_1)} \right. \\ 
  \left. +\sum_{b_1'=0}^{b_i} \sum_{b_2'=0}^{b_j} \sum_{f_1=0}^{\tilde{f}}
     \frac{(-1)^{f_1 + b_1' + b_2'}\binom{\tilde{f}}{f_1}\binom{b_i}{b_1'}\binom{b_j}{b_2'}
    \mu_{2 c_i + 2 c_j - 2 c_1' - 2 c_2' + b_1' + b_2' + 2, 
     k - 1 + f_1 + 2 c_1' + 2 c_2'} }{\zeta_{i,j,c_1',c_2',b_1',b_2'}(k + 2 c_1' + 2 c_2' - 2)! (k - 1 + 2 c_1' + 2 c_2' + f_1)}\right)
\end{gather*}
where $\tilde{d}=b_i + b_j + 2 c_i + 2 c_j - 2 c_1' - 2 c_2' + 2$, $\tilde{e}=b_i - b_1' + b_j + 2 c_j - 2 c_2' + 1$, $\tilde{f}= b_i + b_j - b_1' - b_2'$,
\[\delta_{i,j,c_1',c_2'} = \frac{b_i! b_j! (2 c_i - 2 c_1')! (2 c_j - 
     2 c_2')!}{(b_i + 2 c_i - 2 c_1' + 1)! (b_j + 2 c_j - 
       2 c_2' + 1)!},\]
       \[\varepsilon_{i, j, c_1', c_2', b_1'} = \frac{b_j! (2 c_j - 
      2 c_2')!}{(b_j + 2 c_j - 2 c_2' + 1)!(b_1' + 2 c_i - 2 c_1' + 
       1)},\]
       \[\zeta_{i,j,c_1',c_2',b_1',b_2'} = (b_1' + 2 c_i - 2 c_1' + 1)(b_2' + 2 c_j - 2 c_2' + 1),\]
       \[\lambda_n = \frac{\theta}{2}\left(\frac{\theta}{2} \frac{{}_2F_1(1, 1; n + 2; \theta/2)}{n + 1} -
    H_n + \log\left(1 - \frac{\theta}{2}\right)\right),\]
and
       \[\mu_{m,n}=\frac{\theta}{2}
 \frac{{}_2F_1(1, m; m + n + 1; \theta/2)}{m \binom{m + n}{n}}.\]
 Here ${}_2F_1(a,b;c;z)$ denotes the standard hypergeometric function and $H_n$ denotes the $n$th harmonic number.
\end{theorem}
\begin{proof}
First, we make a change of variables $x_0 = 2\xi/\theta$ and $\tilde{x}_1=x_0+(1-x_0)x_1$ for the inner integral of $F_{\xi}$. Then $L_k(F)$ becomes
\begin{align*}
\int_{2\eta/\theta}^1 \frac{1}{x_0 (\frac{2}{\theta}-x_0)} \mathop{\idotsint}_{\mathcal{R}'_{k-1}} \left( \int_{x_0}^{1-P_1'} P \, d\tilde{x}_1 \right) \left( \int_0^{1-P_1'} P \, dx_1 \right) \,dx_2\cdots dx_k dx_0    
\end{align*}
where $P_1'=x_2+x_3+\cdots +x_k$ and $\mathcal{R}'_{k-1}$ is the region of points $(x_2, x_3, \ldots, x_k)$ such that $x_i\geq 0$ for all $i$ and $x_0 + P_1'\leq 1$. Likewise $M_k(F)$ becomes
\begin{align*}
\frac{\theta}{2}\int_{2\eta/\theta}^1 \frac{1}{x_0 (\frac{2}{\theta}-x_0)} \mathop{\idotsint}_{\mathcal{R}'_{k-1}} \left( \int_{x_0}^{1-P_1'} P \, d\tilde{x}_1 \right)^2 \,dx_2\cdots dx_k dx_0.
\end{align*}
To evaluate the inner integrals we use the binomial theorem on $P_2^{c_i}$, another substitution $u=\tilde{x}_1/(1-P_1')$, and the beta function identity to get
\begin{align*}
&\int_{x_0}^{1-P_1'} P \, d\tilde{x}_1 = \sum_{i=0}^n a_i \int_{x_0}^{1-P_1'} (1-P_1)^{b_i}P_2^{c_i} \, d\tilde{x}_1 \\
=& \sum_{i=0}^n a_i \sum_{c_1'=0}^{c_i} \binom{c_i}{c_1'} (P_2')^{c_1'} \int_{x_0}^{1-P_1'} (1-P_1)^{b_i}\tilde{x}_1^{2c_i-2c_1'} \, d\tilde{x}_1 \\
=& \sum_{i=0}^n a_i \sum_{c_1'=0}^{c_i} \binom{c_i}{c_1'} (1-P_1')^{b_i+2c_i-2c_1'+1} (P_2')^{c_1'} \int_{x_0/(1-P_1')}^1 (1-u)^{b_i}u^{2c_i-2c_1'} \, du \\
=& \sum_{i=0}^n a_i \sum_{c_1'=0}^{c_i} \binom{c_i}{c_1'} \left(\frac{b_i!(2c_i-2c_1')!}{(b_i+2c_i-2c_1'+1)!} (1-P_1')^{b_i+2c_i-2c_1'+1} (P_2')^{c_1'}\right. \\
&\left. - \sum_{b_1'=0}^{b_i} \frac{(-1)^{b_1'}\binom{b_i}{b_1'}x_0^{b_1'+2c_i-2c_1'+1}}{b_1'+2c_i-2c_1'+1} (1-P_1')^{b_i-b_1'}(P_2')^{c_1'} \right)
\end{align*}
This argument also implies
\begin{align*}
    \int_{0}^{1-P_1'} P \, dx_1 = \sum_{j=0}^n a_j \sum_{c_2'=0}^{c_j} \binom{c_j}{c_2'} \frac{b_j!(2c_j-2c_2')!}{(b_j+2c_j-2c_2'+1)!} (1-P_1')^{b_j+2c_j-2c_2'+1} (P_2')^{c_2'},
\end{align*}
so
\begin{align*}
    &\left( \int_{x_0}^{1-P_1'} P \, d\tilde{x}_1 \right) \left( \int_0^{1-P_1'} P \, dx_1 \right) \\
    =& \sum_{0\leq i,j \leq n} a_i a_j \sum_{c_1'=0}^{c_i}\sum_{c_2'=0}^{c_j}\binom{c_i}{c_1'}\binom{c_j}{c_2'}\left( \delta_{i,j,c_1',c_2'} (1-P_1')^{\tilde{d}}(P_2')^{c_1'+c_2'}\right. \\
    -& \left. \sum_{b_1'=0}^{b_i} (-1)^{b_1'}\binom{b_i}{b_1'}\varepsilon_{i, j, c_1', c_2', b_1'} x_0^{b_1'+2c_i-2c_1'+1} (1-P_1')^{\tilde{e}}(P_2')^{c_1'+c_2'}\right),
\end{align*}
and
\begin{align*}
    &\left( \int_{x_0}^{1-P_1'} P \, d\tilde{x}_1 \right)^2 \\
    =& \sum_{0\leq i,j \leq n} a_i a_j \sum_{c_1'=0}^{c_i}\sum_{c_2'=0}^{c_j}\binom{c_i}{c_1'}\binom{c_j}{c_2'}\left( \delta_{i,j,c_1',c_2'} (1-P_1')^{\tilde{d}}(P_2')^{c_1'+c_2'}\right. \\
    -& \left. 2\sum_{b_1'=0}^{b_i} (-1)^{b_1'}\binom{b_i}{b_1'}\varepsilon_{i, j, c_1', c_2', b_1'} x_0^{b_1'+2c_i-2c_1'+1} (1-P_1')^{\tilde{e}}(P_2')^{c_1'+c_2'}\right.\\
    +&\left. \sum_{b_1'=0}^{b_i} \sum_{b_2'=0}^{b_j}
     \frac{(-1)^{b_1' + b_2'}\binom{b_i}{b_1'}\binom{b_j}{b_2'}x_0^{2 c_i + 2 c_j - 2 c_1' - 2 c_2' + b_1' + b_2' + 2} }{\zeta_{i,j,c_1',c_2',b_1',b_2'}} (1-P_1')^{\tilde{f}}(P_2')^{c_1'+c_2'} \right).
\end{align*}
Next we use homogeneity and then apply \cref{Glem} to obtain the intermediate result
\begin{align*}
   &\mathop{\idotsint}_{\mathcal{R}'_{k-1}} (1-P_1')^{\tilde{b}}(P_2')^{c} \,dx_2\cdots dx_k \\
   =& \sum_{b'=0}^{\tilde{b}}(-1)^{b'}\binom{\tilde{b}}{b'} \mathop{\idotsint}_{\mathcal{R}'_{k-1}} (P_1')^{b'}(P_2')^{c} \,dx_2\cdots dx_k \\
   =&\sum_{b'=0}^{\tilde{b}}(-1)^{b'}\binom{\tilde{b}}{b'} (1-x_0)^{k-1+b'+2c} \mathop{\idotsint}_{\mathcal{R}_{k-1}} (P_1')^{b'}(P_2')^{c} \,dx_2\cdots dx_k \\
   =&\sum_{b'=0}^{\tilde{b}} \frac{(-1)^{b'}\binom{\tilde{b}}{b'} Q_c(k-1)}{(k+2c-2)!(k-1+2c+b')}(1-x_0)^{k-1+b'+2c}. \\
\end{align*}
Combining the previous three equations yields
\begin{gather*}
L_k(F) = \sum_{0\leq i,j \leq n} a_i a_j \sum_{c_1'=0}^{c_i}\sum_{c_2'=0}^{c_j}\binom{c_i}{c_1'}\binom{c_j}{c_2' } Q_{c_1' + c_2'}(k - 1) \\
 \cdot \left(\sum_{d_1=0}^{\tilde{d}}  \frac{(-1)^{d_1}\binom{\tilde{d}}{d_1}\delta_{i, j, c_1', c_2'} \displaystyle{\int_{2\eta/\theta}^1 \frac{(1-x_0)^{k-1+d_1+2c_1'+2c_2'}}{x_0 (\frac{2}{\theta}-x_0)} \, dx_0}}{(k + 2 c_1' + 2 c_2' - 2)! (k - 1 + 2 c_1' + 2 c_2' + d_1)} \right.\\
\left.-\sum_{b_1'=0}^{b_i}\sum_{e_1=0}^{\tilde{e}} \frac{(-1)^{e_1 + b_1'}\binom{\tilde{e}}{e_1}\binom{b_i}{b_1'}\varepsilon_{i, j, c_1', c_2', b_1'} \displaystyle{\int_{2\eta/\theta}^1 \frac{(1-x_0)^{k-1+e_1+2c_1'+2c_2'}}{ x_0^{-b_1'-2c_i+2c_1'}(\frac{2}{\theta}-x_0)} \, dx_0}}{(k + 2 c_1' + 2 c_2' - 2)!(k - 1 + 2 c_1' + 2 c_2' + e_1)} \right).
 \end{gather*}
 and
\begin{gather*}
  M_k(F) = \frac{\theta}{2}\sum_{0\leq i,j \leq n} a_i a_j \sum_{c_1'=0}^{c_i}\sum_{c_2'=0}^{c_j}\binom{c_i}{c_1'}\binom{c_j}{c_2'} Q_{c_1'+c_2'}(k-1)\\
  \cdot \left(  \sum_{d_1=0}^{\tilde{d}} \frac{(-1)^{d_1}\binom{\tilde{d}}{d_1}\delta_{i,j,c_1',c_2'} \displaystyle{\int_{2\eta/\theta}^1 \frac{(1-x_0)^{k-1+d_1+2c_1'+2c_2'}}{x_0 (\frac{2}{\theta}-x_0)}  \, dx_0}  }{(k+2c_1'+1c_2'-2)!(k-1+2c_1'+c_2'+d_1)}\right. \\
    - \left. 2\sum_{b_1'=0}^{b_i} \sum_{e_1=0}^{\tilde{e}} \frac{(-1)^{e_1+b_1'}\binom{\tilde{e}}{e_1} \binom{b_i}{b_1'}\varepsilon_{i, j, c_1', c_2', b_1'} \displaystyle{\int_{2\eta/\theta}^1 \frac{(1-x_0)^{k-1+e_1+2c_1'+2c_2'}}{x_0^{-b_1'-2c_i+2c_1'}  (\frac{2}{\theta}-x_0)}  \, dx_0}}{(k+2c_1'+1c_2'-2)!(k-1+2c_1'+c_2'+e_1)} \right.\\
    +\left. \sum_{b_1'=0}^{b_i} \sum_{b_2'=0}^{b_j} \sum_{f_1=0}^{\tilde{f}}
     \frac{(-1)^{f_1+b_1' + b_2'}\binom{\tilde{f}}{f_1}\binom{b_i}{b_1'}\binom{b_j}{b_2'} \displaystyle{\int_{2\eta/\theta}^1 \frac{(1-x_0)^{k-1+f_1+2c_1'+2c_2'} \, dx_0}{x_0^{-2 c_i - 2 c_j + 2 c_1' + 2 c_2' - b_1' - b_2' -1} (\frac{2}{\theta}-x_0)}} }{\zeta_{i,j,c_1',c_2',b_1',b_2'}(k+2c_1'+1c_2'-2)!(k-1+2c_1'+c_2'+f_1)} \right).
\end{gather*}
To evaluate the integrals in $x_0$ we make use of the formula
\[\frac{{}_2F_1(a,b;c;z)}{b\binom{c-1}{b}} = \int_0^1\frac{t^b (1-t)^{c-b-1}}{t(1-zt)^a}\, dt\]
which is valid when $z$ is non-real or strictly less than 1 and where $b$, $c$ are integers with $c>b>0$. Immediately, we get
\[\lim_{\eta\rightarrow 0^+} \int_{2\eta/\theta}^1\frac{x_0^m (1-x_0)^n}{x_0(\frac{2}{\theta}-x_0)}\, dx_0=\frac{\theta}{2}\frac{{}_2F_1(1,m;m+n+1;\theta/2)}{m\binom{m+n}{m}} = \mu_{m,n}\]
when $m>0$, $n\geq 0$ are integers. Next we evaluate
\begin{align*}
    \int_{2\eta/\theta}^1 \frac{1}{x_0(\frac{2}{\theta}-x_0)}\, dx_0 &=  \frac{\theta}{2}\int_{2\eta/\theta}^1 \frac{1}{x_0}\, dx_0 +\frac{\theta}{2}\int_{2\eta/\theta}^1 \frac{1}{\frac{2}{\theta}-x_0}\, dx_0 \\
    &= -\frac{\theta}{2}\log 2\eta/\theta +\frac{\theta}{2}\log\left( \frac{2/\theta-2\eta/\theta}{2/\theta-1}\right) \\
    &= \frac{\theta}{2} \log\left(\frac{1-\eta}{\eta(\frac{2}{\theta}-1)}\right).
\end{align*}
Lastly,
\begin{align*}
    &\lim_{\eta\rightarrow 0^+} \int_{2\eta/\theta}^1\frac{ (1-x_0)^n-1}{x_0(\frac{2}{\theta}-x_0)}\, dx_0 = \frac{\theta}{2}\int_{0}^1\frac{ (1-x_0)^n-1}{x_0(1-\frac{\theta}{2}x_0)}\, dx_0 \\
    &=\frac{\theta}{2}\left(\int_{0}^1 \frac{\frac{\theta}{2}x_0(1-x_0)^n}{x_0(1-\frac{\theta}{2}x_0)} \, dx_0 + \int_{0}^1 \frac{(1-x_0)^n-1}{x_0} \, dx_0 -\int_{0}^1 \frac{\frac{\theta}{2}}{1-\frac{\theta}{2}x_0} \, dx_0  \right) \\
    &=\frac{\theta}{2}\left(\frac{\theta}{2}\frac{{}_2F_1(1,1;n+2;\theta/2)}{n+1} -H_n +\log\left(1-\frac{\theta}{2}\right)  \right) =\lambda_n.
\end{align*}
\end{proof}

\section{Proof of \texorpdfstring{\cref{maintheorem}}{Theorem 1}}

In this last section, we use \cref{tildetheorem} to establish \cref{maintheorem}, which gives conditional and unconditional bounds on gaps for $\nu+1$ successive sifted $E_2$-numbers as seen in \cref{E2table}.

\begin{proof}[Proof of \cref{maintheorem}]
Assume the symmetric polynomial $P$ defining $F$ satisfies \cref{assume}.
As Maynard noted in \cite{Mayn}, $I_k(F)$ and $J_k(F)$ can be regarded as positive definite quadratic forms with $I_k(F)=\a^T A_I \a$ and $J_k(F) = \a^T A_J \a$ where $A_I$, $A_J$ are real symmetric matrices with coefficient vector $\a=(a_0, a_1, \ldots, a_n)$.
We also take $a_m=0$ when $m>n$ for convenience.
Note that
\begin{align*}
        0 \leq &\int_{\eta}^{\theta/2} \frac{1}{\xi (1-\xi)} \int_0^1 \cdots \int_0^1 \left(\frac{\theta}{2}\int_0^1 F \, dx_1 - \left(\frac{\theta}{2}-\xi\right)\int_0^1 F_{\xi} \, dx_1 \right)^2 dx_2\cdots dx_k d\xi \\
        =& -\theta L_k(F) + \frac{\theta^2}{4} \log \left(\frac{1 - \eta}{\eta}\right) J_k(F) +  M_k(F) - \frac{\theta^2}{4}\log\left(\frac{2}{\theta}-1\right)J_k(F),
\end{align*}
so
\begin{align*}
\widetilde{J}_k(F) \geq \frac{\theta^2}{4}\log\left(\frac{2}{\theta}-1\right)J_k(F)>0.
\end{align*}
Thus $\widetilde{J}_k(F)$ can also be regarded as positive definite quadratic forms in the $a_i$ variables.
In particular, $\widetilde{J}_k(F) = \a^TA_{\widetilde{J}}\a$ for some real symmetric matrix $A_{\widetilde{J}}$.
The ratio
\[R_k(F) \mathrel{\mathop :}= \frac{\widetilde{J}_k(F)}{\frac{\theta}{2}I_k(F)}=\frac{2}{\theta}\cdot\frac{\a^TA_{\widetilde{J}}\a}{\a^T A_I \a}\]
is maximized when $\a$ is an eigenvector for the largest eigenvalue of $A_I^{-1}A_{\widetilde{J}}$.
We use Mathematica to find a numerical approximation to this eigenvector and then rationalize it.
The rational vector $\a$ then gives an exact value for $R_k(F)$.
Recall that Sono's \cref{sonotheorem} states that if $\mathcal{H}$ is admissible and $R_k(F)>\nu$, then there are infinitely many integers $n$ such that at least $\nu + 1$ of the $n+h_i$ are sifted $E_2$-numbers, and thus $\widetilde{G}_\nu \leq H(k)$. 
We fix the exponents in a symmetric polynomial $P = \sum_{i=0}^n a_i(1-P_1)^{b_i}P_2^{c_i}$ as follows:
\begin{align*}
    (b_0,b_1, \ldots)=&(0, 1, 2, 0, 3, 1, 4, 2, 0, 5, 3, 1, 6, 4, 2, 0, 7, 5, 3, 1, 8, \\
    &\, \, 6, 4, 2, 0, 9, 7, 5, 3, 1, 10, 8, 6, 4, 2, 0, 11, 9, 7, 5, 3, 1, \ldots)
\end{align*}
\begin{align*}
(c_0,c_1, \ldots)=&(0, 0, 0, 1, 0, 1, 0, 1, 2, 0, 1, 2, 0, 1, 2, 3, 0, 1, 2, 3, 0, \\
&\, \, 1, 2, 3, 4, 0, 1, 2, 3, 4, 0, 1, 2, 3, 4, 5, 0, 1, 2, 3, 4, 5, \ldots)
\end{align*}
We will only derive the entries in the last two columns of \cref{E2table} from the third row and beyond since the first two columns follow directly from the Hardy-Littlewood conjectures and the first two rows in the last two columns were known previously \cite{Gold1,Sono,Gold3}.

We first assume the Elliott-Halberstam conjecture for primes and sifted $E_2$-numbers, i.e., we assume $\mathrm{BV}[\theta,\mathcal{P}]$ and $\mathrm{BV}[\theta,\widetilde{\mathcal{E}}_2]$ with $\theta=1$. When $k=10$ we take $\a =$
\[ \left(-\frac{2301604465403391652}{124720775947120337501}, -\frac{83833247885835453802}{83847526334133337873}\right)\]
which gives $R_{10}(F) = 3.0353\ldots >3$, so $\widetilde{G}_3 \leq H(10)=32$.
When $k=16$ we take $\a =$
\begin{align*}
&\left(\frac{414930787723574773}{109422307624908766578}, 
-\frac{133690150982074747}{35908140287706839643}, \right.\\
&\,\,\,\,\left.-\frac{56372796939984330619}{56516941370382300823}, 
-\frac{1344287896322269883}{18886517158832273469}\right)
\end{align*}
which gives $R_{16}(F) = 4.000399\ldots > 4$, so $\widetilde{G}_4 \leq H(16)=60$.
When $k=25$ we take $\a =$
\begin{align*}
&\left(\frac{1119449503899613}{53442936550267755953},-\frac{10278507226808306}{142013518853788110707},-\frac{5997977329113581403}{155572764212654119585},\right.\\
&\,\,\,\,-\frac{68226903858709892}{54862996313344202931},-\frac{1768150620457940078}{47081978193596417231},\frac{488025969629313828}{190329612086595601535},\\
    &\,\,\,\,\left.\frac{92289937653
   529035180}{148092942398709354329},\frac{69817571067066097358}{89503691073568190133},\frac{1198016722578081977}{76012521424722835124}\right)
\end{align*}
which gives $R_{25}(F) = 5.0454\ldots > 5$, so $\widetilde{G}_5 \leq H(25)=110$.
When $k=37$ we take $\a =$
\begin{align*}
  &\left(-\frac{21148272786657}{97959366205223913548},\frac{158666306211186}{42114599604634278673},-\frac{29189013259110813}{62481776334191804879},\right.\\
  &\,\,\,\,\frac{668709179903374}{98195305652136904677},-\frac{210824068030837527}{112495425279539043386},-\frac{8711811183996893}{117875551469631705341},\\
  &\,\,\,\,\frac{2204822561042490655}{87523836340436194254},\frac{2659497759715353463}{113400244533928417732},\frac{746783608810176}{32877469237109587567},\\
  &\,\,\,\,\frac{6297052498420029709}{107036774210681509668},\frac{6808672822036122389}{105349407944676030889},-\frac{9244669583898397}{64296513546842911844},\\
  &\,\,\,\,-\frac{18172970779893736836}{60401919992146382861},-\frac{94704601442000848643}{105973862042865144084},-\frac{20860537199538498795}{65309953324114639423},\\
  &\,\,\,\,\left.-\frac{91092595669759901}{56609899602147701511}\right)
\end{align*}
which gives $R_{37}(F) = 6.01020\ldots > 6$, so $\widetilde{G}_6 \leq H(37)=168$.

We now prove our unconditional results, so we have $\mathrm{BV}[\theta,\mathcal{P}]$ and $\mathrm{BV}[\theta,\widetilde{\mathcal{E}}_2]$ with $\theta=1/2$.
When $k=23$, we take $\a=$
\begin{align*}
    &\left(-\frac{16068472196488}{57228638216292482079},-\frac{3190274818823462}{
    153789629695933792839},\frac{32660633679801409}{132753576873497845795},\right.\\
    &\,\,\,\,\frac{695244807497850}{60265813892320495919},-\frac{213821666582879462}{127170093
    693670702511},\frac{79021830385258773}{115850080002404163316},\\
    &\,\,\,\,-\frac{7197058
    8270569106}{74380447200009281471},-\frac{401848789372287506}{890924321964173
    13923},-\frac{18212382651929877}{199652089547425152721},\\
    &\,\,\,\,\frac{520156890779156389}{19253968794695459975},\frac{10432966101974650360}{20007748812639916699
    3},-\frac{304748829126875187}{53355025095470545801},\\
    &\,\,\,\,\frac{1722198035704162915}{145031010654506068979},\frac{6230701253494717243}{159260535299935884030},
    \frac{714109959113406671}{85511027875588422495},\\
    &\,\,\,\,-\frac{32115347349591259}{67
    743025151737099673},-\frac{8989436245766664859}{66755896357240274643},-\frac
    {92753168602654697319}{118950522527495106883},\\
    &\,\,\,\,\left.-\frac{62296007634989205265}{102621101829271040879},-\frac{1081831880065143203}{120385336940627151139}\right)
\end{align*}
which gives $R_{23}(F) = 3.00050654254\ldots > 3$, so $\widetilde{G}_3 \leq H(23)=94$.
When $k=49$, we take $\a=$
\begin{align*}
    &\left(\frac{470254915649}{150765533574864473017},\frac{58858471570404}{103001238314402753165},\frac{1810150485536251}{84223196442630167337},\right.\\
    &\,\,\,\,-\frac{50254705151197}{131826793466385769000},-\frac{46495396759026092}{95943844082109607387},-\frac{7117430940522807}{111194365619194158095},\\
    &\,\,\,\,-\frac{186468266732965278}{99349079194135507801},-\frac{90329128207054573}{74639912346406311005},\frac{3592961826811194}{248254292207525722027},\\
    &\,\,\,\,\frac{3582943779810039096}{269497695369277757771},\frac{4415817421641079238}{144860007068392667251},\frac{268031516354274315}{118239374598256503698},\\
    &\,\,\,\,\frac{6872659514750980001}{108233535286137814372},\frac{18520442052690804867}{190854707294436863261},\frac{2299708104592911281}{124235638427450849950},\\
    &\,\,\,\,-\frac{34267710839976933}{194982267383830416128},-\frac{23310239442223830032}{105543537512637134013},-\frac{43379988238424925840}{54421587852935924737},\\
    &\,\,\,\,\left.-\frac{62334839052322588337}{113763328332070308845},-\frac{1702360154187240957}{64471126631620679318}\right)
\end{align*}
which gives $R_{49}(F) = 4.00096634233\ldots > 4$, so $\widetilde{G}_4 \leq H(49)= 240$.
When $k=102$, we take $\a=$
\begin{align*}
   &\left(-\frac{882233146}{1139042176880719802317},-\frac{161821289743}{9998047
    73957773385326},-\frac{2607084026483}{114581766205785386057},\right.\\
    &\,\,\,\,\frac{443008115
    90}{226430413350338959437},\frac{47245088758784}{89482935901465569711},\frac
    {2463128808672}{49142705373843987011},\\
    &\,\,\,\,\frac{383708250626791}{250347662872812
    34113},\frac{281090342737732}{70565972830750945531},-\frac{2580257542838}{13
    9623938130044468817},\\
    &\,\,\,\,-\frac{25970151895414994}{249757311805137989255},-\frac
    {6245935132576615}{54878844481503296056},-\frac{734253786219396}{13713546743
    1277494703},\\
    &\,\,\,\,-\frac{81930822909258153}{74101640871227538557},-\frac{153992009
    708274910}{76919874292267814749},-\frac{20756487275608320}{88162592296370600
    123},\\
    &\,\,\,\,\frac{93845833779892}{121456342703508411095},\frac{210008201228787809}{
    115122764579812246811},\frac{1700637882880385589}{106825778250695241887},\\
    &\,\,\,\,\frac{1422926881003116425}{175684514660828716683},\frac{23641349956090588}{9740
    2228553403995981},\frac{8838029905873008892}{191027996906978661745},\\
    &\,\,\,\,\frac{97
    00352674133921763}{84082913893918959293},\frac{9972940237541621860}{14130703
    8860775976009},\frac{252539609823204694}{53185506895705874279},\\
    &\,\,\,\,-\frac{168100
    287196548}{13944338424533849939},-\frac{8241924226972183000}{628855272944800
    31429},-\frac{39930616213199365077}{67009050047324008132},\\
    &\,\,\,\,\left.-\frac{94093444717
    554618011}{124799006567988815037},-\frac{10644221080155570763}{5426611784503
    4898582},-\frac{145698639501539687}{36196107909744222989}\right)
\end{align*}
which gives $R_{102}(F) = 5.01623513164\ldots > 5$, so $\widetilde{G}_5 \leq H(102)= 576$.

When $k=225$ we take $\a=$
\begin{align*}
 &\left(\frac{1737257}{95954543521478103735},\frac{1255460898}{69116135215736108219},\frac{118729352089}{103666060477672725912},\right.\\
&\,\,\,\,-\frac{2360932799}{276082226049536939879},-\frac{3794174642695}{60254001723992180064},-\frac{135555397041}{15222266869986803038},\\
&\,\,\,\,-\frac{100839561845411}{57393969844387345619},-\frac{30207171207808}{70897404199127479341},\frac{183690703052}{121618923558261209347},\\
&\,\,\,\,\frac{1683993376733374}{79599001542627410941},\frac{3994726767711341}{156079494715224312915},\frac{1019368532630026}{623356860146201150965},\\
&\,\,\,\,\frac{40126853165073949}{89089847694352722874},\frac{60752656221795202}{125336724400151819393},\frac{3617665576441182}{68127557639016623939},\\
&\,\,\,\,-\frac{2594523081261}{21948908931226042559},-\frac{141935015071364327}{74034597470127700608},-\frac{1252235758489804142}{169772047220954666383},\\
&\,\,\,\,-\frac{327635187034772711}{93721827846594798922},-\frac{5992625674476809}{44917761568520489095},-\frac{2594535248830902098}{62603952620982614591},\\
&\,\,\,\,-\frac{7675220155042301711}{98432436539321022970},-\frac{3563161181500469559}{104232117951925073017},-\frac{210051438309136231}{94568269398642645891},\\
&\,\,\,\,\frac{290693037557815}{84062798134792976444},\frac{18205523506808020735}{88231060093766156417},\frac{157487179721487145901}{226894860616830479587},\\
&\,\,\,\,\left.\frac{49527623795301965830}{74610537934569571063},\frac{18609537355041853051}{115537712162976964912},\frac{409442836501323153}{100258685501749588340}\right)
\end{align*}
which gives $R_{225}(F) = 6.0098418048817\ldots > 6$, so $\widetilde{G}_6 \leq H(225)= 1440$.

\end{proof}

\providecommand{\bysame}{\leavevmode\hbox to3em{\hrulefill}\thinspace}
\providecommand{\MR}{\relax\ifhmode\unskip\space\fi MR }
% \MRhref is called by the amsart/book/proc definition of \MR.
\providecommand{\MRhref}[2]{%
  \href{http://www.ams.org/mathscinet-getitem?mr=#1}{#2}
}
\providecommand{\href}[2]{#2}


\begin{thebibliography}{GGPY09}

\bibitem[Bom87]{Bomb}
Enrico Bombieri, \emph{Le grand crible dans la th\'{e}orie analytique des
  nombres}, Ast\'{e}risque (1987), no.~18, 103.

\bibitem[EH70]{Elli}
P.~D. T.~A. Elliott and H.~Halberstam, \emph{A conjecture in prime number
  theory}, Symposia {M}athematica, {V}ol. {IV} ({INDAM}, {R}ome, 1968/69),
  Academic Press, London, 1970, pp.~59--72.

\bibitem[GGPY09]{Gold1}
D.~A. Goldston, S.~W. Graham, J.~Pintz, and C.~Y. Y{\i}ld{\i}r{\i}m,
  \emph{Small gaps between products of two primes}, Proceedings of the London
  Mathematical Society \textbf{98} (2009), no.~3, 741--774.

\bibitem[GPS22]{Gold3}
Daniel~A. Goldston, Apoorva Panidapu, and Jordan Schettler, \emph{Small gaps
  between three almost primes and almost prime powers}, Acta Arith.
  \textbf{203} (2022), no.~1, 1--18. \MR{4415993}

\bibitem[GPY09]{Gold0}
D.~A. Goldston, J.~Pintz, and C.~Y. Y{\i}ld{\i}r{\i}m, \emph{Primes in tuples.
  {I}.}, Annals of Mathematics \textbf{179} (2009), no.~2, 819--862.

\bibitem[May15]{Mayn}
James Maynard, \emph{Small gaps between primes}, Ann. of Math. (2) \textbf{181}
  (2015), no.~1, 383--413.

\bibitem[Mot76]{Moto}
Yoichi Motohashi, \emph{An induction principle for the generalization of
  {B}ombieri's prime number theorem}, Proc. Japan Acad. \textbf{52} (1976),
  no.~6, 273--275.

\bibitem[Pol14]{Poly}
D.~H.~J. Polymath, \emph{Variants of the {S}elberg sieve, and bounded intervals
  containing many primes}, Res. Math. Sci. \textbf{1} (2014), Art. 12, 83.
  \MR{3373710}

\bibitem[Son20]{Sono}
Keiju Sono, \emph{Small gaps between the set of products of at most two
  primes}, J. Math. Soc. Japan \textbf{72} (2020), no.~1, 81--118.

\bibitem[Vin65]{Vino}
A.~I. Vinogradov, \emph{The density hypothesis for {D}irichet {$L$}-series},
  Izv. Akad. Nauk SSSR Ser. Mat. \textbf{29} (1965), 903--934.

\bibitem[Zha14]{Zhan}
Yitang Zhang, \emph{Bounded gaps between primes}, Ann. of Math. (2)
  \textbf{179} (2014), no.~3, 1121--1174.

\end{thebibliography}
\end{document}